\documentclass{amsart}
\usepackage{graphics}
\usepackage{amsfonts,amsmath}
\usepackage{amsmath, amsfonts,amssymb,amsopn,amscd,amsthm}
\usepackage[sc]{mathpazo}
\usepackage{mathrsfs}

\begin{document}

 \newtheorem{thm}{Theorem}[section]
 \newtheorem{cor}[thm]{Corollary}
 \newtheorem{lem}[thm]{Lemma}{\rm}
 \newtheorem{prop}[thm]{Proposition}

 \newtheorem{defn}[thm]{Definition}{\rm}
 \newtheorem{assumption}[thm]{Assumption}
 \newtheorem{rem}[thm]{Remark}
 \newtheorem{ex}{Example}
\numberwithin{equation}{section}

\def\x{\mathbf{x}}
\def\h{\mathbf{h}}
\def\by{\mathbf{y}}
\def\bz{\mathbf{z}}
\def\F{\mathcal{F}}
\def\R{\mathbb{R}}
\def\T{\mathbf{T}}
\def\N{\mathbb{N}}
\def\D{\mathcal{D}}
\def\V{\mathbf{V}}
\def\U{\mathbf{U}}
\def\K{\mathbf{K}}
\def\Q{\mathbf{Q}}
\def\M{\mathbf{M}}
\def\oM{\overline{\mathbf{M}}}
\def\O{\mathbf{O}}
\def\C{\mathbb{C}}
\def\P{\mathbf{P}}
\def\hP{\hat{\mathbf{P}}}
\def\Z{\mathbb{Z}}
\def\H{\mathbf{H}}
\def\A{\mathbf{A}}
\def\V{\mathbf{V}}
\def\AA{\overline{\mathbf{A}}}
\def\B{\mathbf{B}}
\def\c{\mathbf{C}}
\def\L{\mathscr{L}}
\def\bS{\mathbf{S}}
\def\H{\mathbf{H}}
\def\I{\mathbf{I}}
\def\Y{\mathbf{Y}}
\def\X{\mathbf{X}}
\def\G{\mathbf{G}}
\def\f{\mathbf{f}}
\def\z{\mathbf{z}}
\def\v{\mathbf{v}}
\def\y{\mathbf{y}}
\def\d{\hat{d}}
\def\bx{\mathbf{x}}
\def\bI{\mathbf{I}}
\def\y{\mathbf{y}}
\def\g{\mathbf{g}}
\def\w{\mathbf{w}}
\def\b{\mathbf{b}}
\def\a{\mathbf{a}}
\def\u{\mathbf{u}}
\def\q{\mathbf{q}}
\def\e{\mathbf{e}}
\def\s{\mathcal{S}}
\def\f{\boldsymbol{f}}
\def\cc{\mathcal{C}}
\def\co{{\rm co}\,}
\def\tg{\tilde{g}}
\def\tx{\tilde{\x}}
\def\tg{\tilde{g}}
\def\tA{\tilde{\A}}

\def\supmu{{\rm supp}\,\mu}
\def\supp{{\rm supp}\,}
\def\cd{\mathcal{C}_d}
\def\cok{\mathcal{C}_{\K}}
\def\cop{COP}
\def\vol{{\rm vol}\,}
\def\erf{{\rm erf}}
\def\erfc{{\rm erfc}}
\def\xsc{\mathscr{X}}
\def\tsc{\mathscr{T}}
\def\p{\mathbf{p}}
\def\g{\boldsymbol{\hat{g}}}
\def\bof{\boldsymbol{f}}
\def\la{\langle}
\def\ra{\rangle}

\title{Lebesgue decomposition in action via semidefinite relaxations}
\author{Jean B. Lasserre}
\thanks{Jean B Lasserre: 7 Avenue du Colonel Roche, BP 54200, 31031 Toulouse cedex 4, France. \\
Tel: +66561336415; Fax: +33561336936; Email: {\tt lasserre@laas.fr}}
\address{LAAS-CNRS and Institute of Mathematics\\
University of Toulouse\\
LAAS, 7 avenue du Colonel Roche, BP 54200\\
31031 Toulouse C\'edex 4, France\\
Tel: +33561336415}
\email{lasserre@laas.fr}

\date{}
\begin{abstract}
Given all (finite) moments of two measures $\mu$ and $\lambda$ on $\R^n$, we provide a numerical scheme to 
obtain the Lebesgue decomposition 
$\mu=\nu+\psi$ with $\nu\ll\lambda$ and $\psi\perp\lambda$. When
$\nu$ has a density in $L_\infty(\lambda)$ then we obtain two sequences of finite moments vectors
of increasing size (the number of moments) which converge to the moments of $\nu$ and $\psi$ respectively, as the 
number of moments increases. Importantly, {\it no} \`a priori knowledge on the supports of $\mu, \nu$ and $\psi$ is required.
\end{abstract}

\maketitle

\section{Introduction}

This work is in the line of research 
concerned with the following issue: {\it which type and how much of information 
on the support of a measure can be extracted from its moments} (a research issue 
outlined in a {\it Problem session} at the 2013 Oberwolfach  meeting 
on {\it Structured Function Systems and Applications} \cite{ober-report}). 

- For instance, the old and classical L-moment problem
asks for moment conditions to ensure that the underlying unknown measure $\mu$ is absolutely continuous with respect to some reference measure $\nu$, and with a density in $L_\infty(\nu)$. See for instance 
Diaconis and Freedman \cite{diaconis}, Putinar \cite{putinar1,putinar2}, more recently Lasserre \cite{density}, and the many references therein. 

- In some other works one is interested in recovering information about the {\it support} of a measure from knowledge of its moments 
and some {\it \`a priori} information on its support. For instance, in Lasserre \cite{bounding}, bounds on the support of a measure are provided from the knowledge of 
moments of its marginals.
In Collowald et al. \cite{collowald} and Gravin et al. \cite{gravin}, the support is assumed to be a convex polytope $P\subset\R^n$ and
the vertices of $P$ can be recovered from finitely many directional moments of the Lebesgue measure on $P$. Similarly,
in Lasserre \cite{dcg-2} and Lasserre and Putinar \cite{dcg} one may recover the boundary of a semi-algebraic set $S\subset\R^n$ from moments of the Lebesgue measure on $S$ (or of a measure with polynomial or exponential of a polynomial density w.r.t. the Lebesgue measure).

- In super-resolution one is interested in recovering the discrete support (a finite set of points) of a signal from finitely many moments of its
associated signed measure; see the ground-breaking work of Cand\`es et Fern\'andez-Granda \cite{candes}, and the recent multivariate extension in  de Castro et al. \cite{decastro}.

- Recently in Lasserre \cite{sigma} one is interested in detecting whether a given measure $d\mu(x,t)$ on $[0,1]^2$ with given marginal $dt$ on $[0,1]$
is supported on a curve $(t,x(t))\subset [0,1]^2$ for some measurable function $x:[0,1]\to [0,1]$. Sufficient and necessary conditions on the moments of $\mu$
are provided in \cite{sigma}.\\

\begin{center}
{\it The present paper is concerned with ``computing" the Lebesgue decomposition 
$\nu+\psi$ of a measure $\mu$ with respect to (w.r.t.) another measure $\lambda$ (i.e., with $\nu\ll\lambda$ and $\psi\perp\lambda$), from the sole knowledge
of moments of $\mu$ and $\lambda$.}
\end{center}

\subsection*{Contribution} Given two measures $\lambda,\mu$ on $\R^n$, this paper provides an effective numerical scheme
to obtain the Lebesgue decomposition $\nu+\psi$ of $\mu$ w.r.t. $\lambda$, where 
$\nu\ll\lambda$ and $\psi\perp\lambda$.
\begin{itemize}
\item The only information that we use is the sequence of
moments $(\lambda_\alpha)$ and $(\mu_\alpha)$, $\alpha\in\N^n$, of $\lambda$ and $\mu$. In particular, {\it no}  \`a priori information on the respective supports of $\lambda$ and $\mu$ is required.
\item The methodology is to treat the problem as an infinite-dimensional convex (and even linear) optimization problem on a appropriate space of measures, and then approximate its solution via a hierarchy of semidefinite relaxations $(\P_d)$, $d\in\N$.
That is, each $\P_d$ is a semidefinite program\footnote{A semidefinite program is a finite-dimensional convex  conic optimization problem which can be solved efficiently (up to arbitrary precision fixed in advance) in time polynomial in its input size.} whose size increases with $d$ (as more moments of $\mu$ and $\lambda$ are taken into account). It also includes a scalar parameter $\gamma>0$, fixed in advance.
\item 
The output is under the form of two finite sequences $(y_\alpha)$ and $(\mu_\alpha-y_\alpha)$, $\alpha\in\N^n_d$, whose length ${n+d\choose n}$
is the number of power moments up to order $d$.
When $\nu\ll\lambda$ has a density $f\in L_\infty(\lambda)$ with $\Vert f\Vert_\infty\leq\gamma$, 
then the two sequences converge (as $d$ increases) to the respective moments of $\nu$ and $\psi$ in the Lebesgue decomposition $\nu+\psi=\mu$.
Otherwise if $\Vert f\Vert_\infty>\gamma$ or if $f\not\in L_\infty(\lambda)$, the 
two sequences converge to the respective moments of $d\nu_\gamma:=(\gamma\wedge f)d\lambda$ and $\psi_\gamma:=\mu-\nu_\gamma$.
\end{itemize}

We do not treat the problem in full generality as 
one obtains the required information on the Lebesgue decomposition $(\nu,\psi)$ only when the density of $\nu$ w.r.t. $\lambda$ is
in $L_\infty(\lambda)$ with norm bounded by $\gamma$, fixed \`a priori. Otherwise one obtains a partial information only. However, so far and to the best of our knowledge, this is the first systematic numerical scheme at this level of generality.

As some (simple) examples treated in this paper show, some numerical issues should be taken into account.
For instance, for clarity of exposition and for convenience of algorithm implementation,
we have chosen power moments associated with the basis of monomials $(\x^\alpha)$, $\alpha\in\N$.
This is typically a bad choice from a numerical viewpoint, especially as we use semidefinite solvers for which such issues can be crucial even for relatively small size matrices.
For instance, a better choice would be a basis of polynomials orthogonal w.r.t. to $\mu$ and/or w.r.t. $\lambda$ (we explain later
how both bases can be used). However such issues are beyond the scope of the present paper devoted to the basic methodology.

\section{Lebesgue decomposition and convex optimization}

\subsection{Notation and definition}
Given a Borel set $\X\subset\R^n$ denote by $M(\X)$ the Banach space of finite signed measures on $\X$, equipped with the total variation norm, and denote by $M(\X)^*$ its topological dual. The notation $M(\X)_+\,(\subset M(\X))$ stands for the convex cone of finite positive measures on $\X$. Denote by $B(\X)$ the Banach space of bounded measurable functions on $\X$, equipped with the sup-norm $\Vert f\Vert_\infty=\sup_{\x\in \X}\vert f(\x)\vert$, and by $B(\X)_+$ the convex cone of nonnegative elements of $B(\X)$.

Let $L_1(\X,\lambda)$ denote the Lebesgue space of $\lambda$-integrable functions, a Banach space when  equipped with the norm
$\Vert f\Vert_1=\int \vert f\vert d\lambda$. Its topological dual is the Lebesgue space $L_\infty(\X,\lambda)$ of functions 
whose essential supremum $\Vert f\Vert_\infty$ (w.r.t. $\lambda$) is finite. The notation $L_1(\X,\lambda)_+$ (resp. $L_\infty(\X,\lambda)_+$) stands for the convex cone of the nonnegative elements of $L_1(\X,\lambda)$ (resp. of $L_\infty(\X,\lambda)$).

Given two Borel measures $\mu,\lambda\in M(\X)_+$ there is a unique Lebesgue decomposition $\mu=\nu+\psi$ with $\nu,\psi\in M(\X)_+$
and $\nu\ll\lambda$, $\psi\perp\lambda$. The notation $\nu\ll\lambda$ means that $\nu$ is absolutely continuous with respect to (w.r.t.) 
$\lambda$ whereas the notation $\psi\perp\lambda$ means that $\psi$ is singular w.r.t. $\lambda$.
Given $\lambda\in M(\X)_+$, the set $C_\lambda\subset M(\X)_+$ defined as:
\[C_\lambda\,:=\,\{\,\nu\in M(\X)_+:\: \nu\,\ll\,\lambda\,\},\]
is a convex cone. If one considers the {\it dual pair} of vector spaces $(M(\X),B(\X))$ with duality bracket
\[\langle \nu,f\rangle \,=\,\int_\X f\,d\nu,\qquad (\nu,f)\,\in\,\M(\X)\times B(\X),\]
the dual cone of $C_\lambda$ denoted $C_\lambda^*\subset B(\X)$ is defined by:
\begin{eqnarray*}
C_\lambda^*&:=&\{\,f\in B(\X):\:\langle \nu,f\rangle \,\geq\,0,\quad\forall\nu\in C_\lambda\,\}\\
&=&\{\,f\in B(\X):\: \int_\X f\,h\,d\lambda \geq0,\quad\forall h\in L_1(\X,\lambda)_+\,\}\,\simeq\,L_\infty(\X,\lambda)_+.
\end{eqnarray*}

Let $\R[\x]$ be the ring of polynomials in the variables
$\x=(x_1,\ldots,x_n)$.
Denote by $\R[\x]_d\subset\R[\x]$ the vector space of
polynomials of degree at most $d$, which forms a vector space of dimension $s(d)={n+d\choose d}$, with e.g.,
the usual canonical basis $(\x^\alpha)$, $\alpha\in\N^n$, of monomials.
Also, let $\N^n_d:=\{\alpha\in\N^n\,:\,\sum_i\alpha_i\leq d\}$ and
denote by $\Sigma[\x]\subset\R[\x]$ (resp. $\Sigma[\x]_d\subset\R[\x]_{2d}$)
the space of sums of squares (SOS) polynomials (resp. SOS polynomials of degree at most $2d$). 
If $f\in\R[\x]_d$, write
\[f(\x)\,=\,\sum_{\alpha\in\N^n_d}f_\alpha\, \x^\alpha\quad\left(= \sum_{\alpha\in\N^n_d}f_\alpha \,\x_1^{\alpha_1}\cdots \x_n^{\alpha_n}\right),\]
in the canonical basis and
denote by $\f=(f_\alpha)\in\R^{s(d)}$ its vector of coefficients. Finally, let $\s^n$ denote the space of 
$n\times n$ real symmetric matrices, with inner product $\la \A,\B\ra ={\rm trace}\,\A\B$, and where the notation
$\A\succeq0$ (resp. $\A\succ0$) stands for $\A$ is positive semidefinite. 

Let $(\A_j)$, $j=0,\ldots,s$,  be a set of real symmetric matrices. An inequality of the form
\[\left(\A(\x)\,:=\,\right)\: \A_0+\sum_{k=1}^s\A_k\,x_k\:\succeq0,\quad \x\in\R^s,\]
is called a {\it Linear Matrix Inequality} (LMI) and a set of the form $\{\x: \A(\x)\succeq0\}$ is
the canonical form of the feasible set of semidefinite programs\footnote{The canonical form of a semidefinite program is
$``\inf \,\{\mathbf{c}^T\x: \A(\x)\succeq0\}$" where $\mathbf{c}\in\R^s$ and $\A_k$, $k=0,\ldots,s$, are real symmetric matrices.}.

Given a real sequence $\z=(z_\alpha)$, $\alpha\in\N^n$, define the Riesz linear functional $L_\z:\R[\x]\to\R$ by:
\[f\:(=\sum_\alpha f_\alpha\x^\alpha)\quad\mapsto L_\z(f)\,=\,\sum_{\alpha}f_\alpha\,z_\alpha,\qquad f\in\R[\x].\]
A sequence $\z=(z_\alpha)$, $\alpha\in\N^n$, has a representing measure $\mu$ if
\[z_\alpha\,=\,\int_{\R^n}\x^\alpha\,d\mu,\qquad\forall\,\alpha\in\N^n.\]

\subsection*{Moment matrix}
The {\it moment} matrix associated with a sequence
$\z=(z_\alpha)$, $\alpha\in\N^n$, is the real symmetric matrix $\M_d(\z)$ with rows and columns indexed by $\N^n_d$, and whose entry $(\alpha,\beta)$ is just $z_{\alpha+\beta}$, for every $\alpha,\beta\in\N^n_d$. 
Alternatively, let
$\v_d(\x)\in\R^{s(d)}$ be the vector $(\x^\alpha)$, $\alpha\in\N^n_d$, and
define the matrices $(\B_\alpha)\subset\s^{s(d)}$ by
\begin{equation}
\label{balpha}
\v_d(\x)\,\v_d(\x)^T\,=\,\sum_{\alpha\in\N^n_{2d}}\B_\alpha\,\x^\alpha,\qquad\forall\x\in\R^n.\end{equation}
Then $\M_d(\z)=\sum_{\alpha\in\N^n_{2d}}z_\alpha\,\B_\alpha$.
If $\z$ has a representing measure $\mu$ then
$\M_d(\z)\succeq0$ because
\[\langle\f,\M_d(\z)\f\rangle\,=\,\int f^2\,d\mu\,\geq0,\qquad\forall \,\f\,\in\R^{s(d)}.\]
In this case 
\begin{equation}
\label{riesz-integral}
L_\z(f)\,=\,\int f\,d\mu,\qquad\forall f\in\R[\x].
\end{equation}

A measure whose all moments are finite is said to be {\it moment determinate} if there is no other measure with the same moments.
The support of a Borel measure $\mu$ on $\R^n$ (denoted $\supmu$) is the smallest closed set $\K$ such that $\mu(\R^n\setminus\K)=0$.

A sequence $\z=(z_\alpha)$, $\alpha\in\N^n$, satisfies Carleman's condition if
\begin{equation}
\label{carleman}
\sum_{k=1}^\infty L_\z(x_i^{2k})^{-1/2k}\,=\,+\infty,\qquad\forall i=1,\ldots,n.
\end{equation}
If a sequence $\z=(z_\alpha)$, $\alpha\in\N^n$, satisfies Carleman's condition (\ref{carleman}) and $\M_d(\z)\succeq0$ for all $d=0,1,\ldots$, then
$\z$ has a representing measure on $\R^n$ which is moment determinate; see e.g. \cite[Proposition 3.5]{lass-book-icp}.
In particular a sufficient condition for a measure $\mu$ to satisfy Carleman's condition 
is that $\int \exp(c\sum_i\vert x_i\vert)d\mu <\infty$ for some $c>0$.

For more details on the above notions as well as their use in potential applications, the interested reader is referred to Lasserre \cite{lass-book-icp}.

\subsection{Lebesgue decomposition as a convex optimization problem}
Given two finite Borel measures $\mu,\lambda\in M(\X)_+$,
consider the infinite-dimensional optimization problem:
\begin{eqnarray}
\nonumber
\P:\quad &\rho=&\displaystyle\sup_\nu\,\{\nu(\X):\nu\leq\mu;\quad \nu\ll\lambda;\:\nu\in M(\X)_+\,\}\\
\label{conv-pb}
&=&\displaystyle\sup_{\nu,\psi}\,\{\nu(\X):\nu+\psi\,=\,\mu;\quad\nu\in C_\lambda,\:\psi\in M(\X)_+\,\},
\end{eqnarray}
where the notation $\nu\leq\mu$ is understood {\it setwise}, i.e., $\nu(B)\leq \mu(B)$ for all Borel sets $B\subset\R^n$.

\begin{thm}
\label{th-abstract-1}
The optimization problem (\ref{conv-pb}) has a unique optimal solution $\nu^*\in M(\X)_+$
and $(\nu^*,\mu-\nu^*)$ provides the Lebesgue decomposition of $\mu$ w.r.t. $\lambda$.
\end{thm}
\begin{proof}
The set $\Delta:=\{\nu\in M(\X)_+:\nu\leq\mu\,\}\subset M(\X)_+$ is not empty. Moreover
it is bounded since $\nu(\X)\leq\mu(\X)$ for all $\nu\in\Delta$. Therefore the countable additivity 
of $\nu$ on $\X$ is uniform with respect to $\nu\in\Delta$. Hence by Dunford \& Schwartz \cite[Theorem 1, p. 305]{dunford}, the set $\Delta$ is weakly sequentially compact. Therefore let
$(\nu_n)\subset\Delta$, $n\in\N$, be a maximizing sequence of (\ref{conv-pb}).
By weak sequential compactness of $\Delta$, there exists 
$\nu^*\in \Delta$ and a subsequence $(n_k)$, $k\in\N$, such that 
\[\lim_{k\to\infty}\int f\,d\nu_{n_k}\,=\,\int f\,d\nu^*,\quad\forall f\in M(\X)^*,\]
and in particular, {\it setwise} convergence takes place, i.e.,
\[\lim_{k\to\infty}\nu_{n_k}(B)\,=\,\nu^*(B),\quad\forall B\in\mathcal{B}(\X).\]
This also implies $\rho=\lim_{k\to\infty}\nu_{n_k}(\X)=\nu^*(\X)$. 
Next, as $\nu_{n_k}\ll\lambda$ for all $k$, a consequence of the above setwise convergence 
is that $\nu^*\ll\lambda$. 
It remains to prove that $\psi^*:=(\mu-\nu^*)\perp\lambda$. 
Assume that this is not the case. Then the Lebesgue decomposition of
$\psi^*$ w.r.t. $\lambda$ yields that $\psi^*=\varphi+\chi$ with $0\neq\varphi\ll\lambda$ and $\chi\perp\lambda$. 
In addition, as $\nu^*+\psi^*=\mu$ we also have $\nu^*+\varphi\leq\mu$.
But then $\tilde{\nu}:=(\nu^*+\varphi)\in\Delta$ and $\tilde{\nu}(\X)=\rho+\varphi(\X)>\rho$,
a contradiction. Therefore $\nu^*\ll\lambda$ and $\psi^*\perp\lambda$ which proves that 
$(\nu^*,\psi^*)$ is the Lebesgue decomposition of $\mu$ w.r.t. $\lambda$. Uniqueness of the latter implies 
that $\nu^*$ is the unique optimal solution of $\P$.
\end{proof}
Observe that $\P$ is a {\it convex} (but infinite dimensional) conic optimization problem, and in fact even an infinite dimensional linear programming problem (or {\it linear program} (LP)). Its dual $\P^*$ is the linear program
\begin{eqnarray}
\nonumber
\P^*:\quad &\rho^*=& \displaystyle\inf_f\,\{\,\int_\X f\,d\mu:\quad  f-1\,\in\,C^*_\lambda;\: f\in B(\X)_+\,\}\\
\label{pb-dual}
&=& \displaystyle\inf_f\,\{\,\int_\X f\,d\mu:\quad  f-1\,\geq\,0\quad\mbox{$\lambda$-a.e.};\quad f\in B(\X)_+\,\},\end{eqnarray}
and by standard weak duality $\rho^*\geq\rho$. In fact:
\begin{lem}
\label{lem-dual}
Let $(\nu^*,\psi^*)$ be the unique optimal solution of $\P$, and let $B^*\in \mathcal{B}(\X)$ be a Borel set such that
$\lambda(B^*)=\lambda(\X)$ and $\psi^*(\X\setminus B^*)=\psi^*(\X)$. Then an optimal solution of $\P^*$ is the function $f^*\in B(\X)_+$  such that 
$f^*(\x)=0$ on $\X\setminus B^*$ and $f^*(\x)=1$ on $B^*$.
\end{lem}
\begin{proof}
The function $f^*$  in Lemma \ref{lem-dual} is feasible for $\P^*$ with associated value
\[\rho^*\,\leq\,\int_\X f^*\,d\mu\,=\,\int_{B^*}f\,d\mu\,=\,\mu(B^*)\,=\,\psi^*(B^*)+\nu^*(B^*)\,=\,\nu^*(\X)\,=\,\rho,\]
and the result follows since $\rho^*\geq \rho$.
\end{proof}
So the two LP (\ref{conv-pb}) and (\ref{pb-dual}) provides us with two dual characterizations of the problem. Namely:
\begin{itemize}
\item An optimal solution $(\nu^*,\psi^*)\in M(\X)_+^2$ of (\ref{conv-pb}) identifies the 
Lebesgue decomposition of $\mu$ w.r.t. $\lambda$.
\item An optimal solution $f^*\in B(\X)$ of (\ref{pb-dual}) is the indicator function $1_{B^*}$ of the 
Borel set $B^*$ on which the restriction of $\mu$ (i.e. $\nu^*$) is absolutely continuous w.r.t. $\lambda$.
\end{itemize}

\subsection{A Lebesgue space version}
In the Lebesgue decomposition $(\nu,\psi)$ (with $\psi:=\mu-\nu$) of $\mu$ w.r.t. $\lambda$ as in Theorem \ref{th-abstract-1},
the measure $\nu$ has a Radon-Nikodym derivative $f\in L_1(\X,\lambda)_+$. 

Next, consider the two dual pairs of vector spaces $(L_1(\X,\lambda),L_\infty(\X,\lambda))$ and
$(M(\X),B(\X))$. Introduce the linear mapping
\[\T :L_1(\X,\lambda)\to M(\X);\quad h\mapsto \T h(B)\,:=\,\int_B h\,d\lambda,\quad\forall B\in\mathcal{B}(\X),\]
which is the canonical embedding of $L_1(\X,\lambda)$ into $M(\X)$. 
The adjoint $\T^*:B(\X)\to L_\infty(\X,\lambda)$ is such that
$\langle \T f,h\rangle=\langle f,\T^*h\rangle$ for all $(f,h)\in L_1(\X,\lambda)\times B(\X)$. That is,
\[\T^*:B(\X)\to L_\infty(\X,\lambda);\quad h\mapsto \T^*h\,:=h,\qquad \forall h\in B(\X),\]
is the natural embedding of $B(\X)$ in $L_\infty(\X,\lambda)$.\\

Now introduce the infinite dimensional linear program (LP):
\begin{eqnarray}
\nonumber
\hat{\P}:\quad \theta&=&\displaystyle\sup_{f}\,\{\int_\X fd\lambda: \T f\,\leq\,\mu;\quad f\in L_1(\X,\lambda)_+\,\}\\
\label{pb-lp}
&=&\displaystyle\sup_{f,\psi}\,\{\int_\X fd\lambda: \T f+\psi=\mu;\quad f\in L_1(\X,\lambda)_+;\quad \psi\in M(\X)_+\,\}.
\end{eqnarray}
Indeed (\ref{pb-lp}) is a conic linear program as $L_1(\X,\lambda)_+\subset L_1(\X,\lambda)$ and $M(\X)_+\subset M(\X)$ are two convex cones. The dual of (\ref{pb-lp}) is the linear program:
\begin{equation}
\label{pb-lp-dual}
\hat{\P}^*:\quad \theta^*\,=\,\displaystyle\inf_{h}\,\{\int_\X h\,d\mu: \T^*h-1\geq0;\quad h\in B(\X)_+\,\}.
\end{equation}
Of course by weak duality one has $\theta\leq \theta^*$ but we also have:
\begin{cor}
\label{th-L1}
An optimal solution of the linear program $\hat{\P}$ in (\ref{pb-lp}) is the Radon-Nikodym derivative $f^*\in L_1(\X,\lambda)_+$ 
of $\nu^*$ w.r.t. $\lambda$ and the optimal solution $h^*\in B(\X)_+$ of $\P^*$ is also an optimal solution of 
$\hat{\P}^*$, so that  $\theta=\theta^*=\rho$.
\end{cor}
\begin{proof}
One  recognizes that
$\hat{\P}^*$ is another phrasing of $\P^*$ so that $\theta^*=\rho^*=\rho$. Concerning (\ref{pb-lp}) we also have
$\theta\geq\rho$ because the Radon Nikodym derivative $f^*$ of $\nu^*$ w.r.t. $\lambda$  in Theorem \ref{th-abstract-1}
is feasible for $\hat{\P}$, with associated value $\int_\X f^*d\lambda=\nu^*(\X)=\rho$. And so from $\theta\leq\theta^*=\rho$ we deduce
that $f^*$ is an optimal solution of $\hat{\P}$.
\end{proof}
One may also observe that the feasible set \[\Delta'\,:=\,\{\,f\in L_1(\X,\lambda):\quad \T f\,\leq\,\mu;\:f\geq 0\,\}\]
of (\ref{pb-lp}) is a convex subset of $L_1(\X,\lambda)$ which by Dunford \& Schwartz , \cite[Theorem 9, p. 292]{dunford}
is weakly sequentially compact. Therefore as in (\ref{pb-lp}) one minimizes a weakly continuous linear functional, there is an optimal solution $f^*\in L_1(\X,\lambda)_+$.\\

So again the two LP (\ref{pb-lp}) and (\ref{pb-lp-dual}) provides us with two dual characterizations of the problem. Namely:
\begin{itemize}
\item An optimal solution $f^*\in L_1(\X,\lambda)_+$ of (\ref{pb-lp}) identifies the 
Radon-Nikodym derivative of $\nu^*$ w.r.t. $\lambda$.
\item An optimal solution $h^*\in B(\X)$ of (\ref{pb-lp-dual}) is the indicator function $1_{B^*}$ of the 
Borel set $B^*$ on which the restriction of $\mu$ (i.e. $\nu^*$) is absolutely continuous w.r.t. $\lambda$.

\end{itemize}

\subsection{Some considerations about possible computation}
Recall that all we know about the problem data is the respective moment sequences of $\mu$ and $\lambda$.
In this context we argue that trying to solve the LP (\ref{pb-lp}) may  not be a good strategy. Indeed by using
that polynomials are dense in $L_1(\X,\lambda)$ when all moments of $\lambda$ exist, one 
is tempted to replace (\ref{pb-lp}) with
\begin{equation}
\label{pb-lp-poly}
\tilde{\P}:\quad \rho'=\displaystyle\sup_{p\in\R[\x]}\,\{\int_\X p\,d\lambda: \T p\,\leq\,\mu;\quad p\in \mathcal{P}(\X)\,\}
\end{equation}
where $\mathcal{P}(\X)$ the convex cone of polynomials that are nonnegative on $\X$. 
Let $f^*\in L_1(\X,\lambda)$ be an optimal solution of $\hat{\P}$ as in Corollary \ref{th-L1}.
If there is a sequence $(p_k)$, $k\in\N$, of polynomials such that $p_k\leq f^*$ for all $k$ and
$\int_\X (f^*-p_k)d\lambda\to 0$ as $k\to\infty$, then both problems $\tilde{\P}$ and $\hat{\P}$ have same optimal value
(but in general $\tilde{\P}$ does not have an optimal solution, except of course if
$\nu$ has a polynomial density). However, even in this case
the difficulty is how to handle the constraint $\T p_k\leq\mu$ (for all $k$) only from the knowledge of the moments of $\mu$.
Therefore in a (hypothetic) maximizing sequence $(p_k)\subset L_1(\X,\lambda)_+$, of (\ref{pb-lp-poly}) 
(where $\T p_k\not\leq \mu$) difficulties arise
for the limit as $k\to\infty$ because if we do no have $\T p_k\leq\mu$ for all $k$ then 
we cannot invoke a setwise convergence argument to show that $\T p_k\to \nu^*$.

An alternative is to work in $M(\X)$ rather than in $L_1(\X,\lambda)$, i.e., to try to solve the 
LP (\ref{conv-pb}). The difficulty is now how to handle the constraint
$\nu\ll \lambda$ only from knowledge of moments of $\lambda$. Fortunately, we know how to do that
if the density $f^*$ of $\nu^*\ll\lambda$ is assumed to be in $L_\infty(\X,\lambda)_+$.  Indeed in this case we may invoke the following result:
\begin{thm}
\label{th-carleman}
Let $\X\subset\R^n$ and let $(\nu_\alpha)$ and $(\lambda_\alpha)$, $\alpha\in\N^n$, be the respective moment sequences 
of a finite Borel measure $\nu$ and $\lambda$ on $\X$. Let $\M_d(\nu)$ and $\M_d(\lambda)$, be their respective moment matrices, $d=0,1,\ldots$. Assume that $(\lambda_\alpha)_{\alpha\in\N^n}$ satisfies Carleman's condition (\ref{carleman}). Then the following two statements are equivalent:

(a) $\nu\ll \lambda$ with density $f\in L_\infty(\X,\lambda)_+$ and $\Vert f\Vert_\infty\leq\gamma$.

(a) $\M_d(\nu)\preceq \gamma\,\M_d(\lambda)$ for every integer $d$ and for some $\gamma>0$.
\end{thm}
\begin{proof}
The implication (a) $\Rightarrow$ (b) is straighforward. Indeed, let $\gamma$ be as in (a). Then:
\[\int_\X g^2\,d\nu\,=\,\int_\X g^2 fd\lambda\,\leq\,\Vert f\Vert_\infty\,\int_\X g^2d\lambda\,\leq\,\gamma\,\int_\X g^2d\lambda
,\qquad \forall\,g\in\R[\x]_d,\]
which shows that $\M_d(\nu)\preceq \gamma\,\M_d(\lambda)$. The reverse implication follows from
\cite[Theorem 3.13]{lass-book-icp}.
\end{proof}

In the next section we will see that the constraint $\M_d(\nu)\preceq\gamma\,\M_d(\lambda)$ is easy to implement
as it is a linear matrix inequality on the unknown moments of $\nu$.

\section{A numerical  approximation scheme}

In this section we will show how to obtain
the Lebesgue decomposition 
$\mu=\nu+\psi$ of $\mu$ w.r.t. $\lambda$, when the absolutely continuous part
$\nu\ll\lambda$ has a density $f\in L_\infty(\X,\lambda)_+$ (instead of $f\in L_1(\X,\lambda)_+$) with $\Vert f\Vert_\infty\leq\gamma$.
We also assume that 
both $\lambda$ and $\mu$ satisfy Carleman's condition (\ref{carleman}) (automatically satisfied when $\X$ is compact).

With this additional restriction we next see that one may indeed provide a numerical scheme 
to approximate the mass of $\nu$, and in fact, any fixed (arbitrary) number of  moments of $\nu$ and $\psi$.
In addition, when  $\psi$ is supported on finitely many atoms, one can 
sometimes extract the support of $\psi$.

In fact, when $\nu\ll\lambda$ with a density $f\not\in L_\infty(\X,\lambda)_+$ or with a density $f\in L_\infty(\X,\lambda)_+$ 
such that $\Vert f\Vert_\infty >\gamma$,
the procedure that we describe below will provide in the limit, the moment sequence of the measure 
$\nu_\gamma\ll\lambda$ with density $f_\gamma=\gamma\wedge f$, so that $f_\gamma\in L_\infty(\X,\lambda)_+$ (but
then of course $\psi=\mu-\nu_\gamma$ is not singular w.r.t. $\lambda$).\\

So, with $\gamma>0$ fixed, introduce the following infinite-dimensional optimization problem:
\begin{equation}
\label{conv-pb-gamma}
\rho_\gamma\,=\,\displaystyle\sup_\nu\,\{\nu(\X):\nu\leq\mu;\quad \nu\leq\gamma\,\lambda;\quad\nu\in M(\X)_+\,\}.
\end{equation}
\begin{thm}
\label{th-abstract-gamma}
Let $(\nu^*,\mu-\nu^*)$ be the Lebesgue decomposition of $\mu$ w.r.t. $\lambda$, and let $f^*\in L_1(\X,\lambda)_+$
be the density of $\nu^*$.
The Borel measure $\nu^*_\gamma\ll\lambda$ with density $f^*_\gamma:=\gamma\wedge f^*$ in $L_\infty(\X,\lambda)_+$
is the unique optimal solution of (\ref{conv-pb-gamma}).
\end{thm}
\begin{proof}
Let $\nu\in M(\X)_+$ be any feasible solution of (\ref{conv-pb-gamma}). Let $B^*$ be a Borel set such that
$\lambda(B^*)=\lambda(\X)$ and $\psi^*(B^*)=0$. From  $\nu\leq\gamma \lambda$ we also have
$\nu\ll\lambda$ and the density $f_\gamma$ of $\nu$ is such that $0\leq f_\gamma\leq \gamma$ on $B^*$.
From $\nu\leq\mu$ we also deduce that
$\nu\leq\nu^*$ on $B^*$, that is (as both $\nu\ll\lambda$ and $\nu^*\ll\lambda$),
$f_\gamma\leq f^*$ a.e. on $\X$, and so $f_\gamma \leq \gamma\wedge f^*$ a.e. on $\X$.
But then $\nu(\X)\leq \int_\X (\gamma\wedge f^*)d\lambda=\nu^*_\gamma(\X)$. Finally, assume that there exists
another optimal solution $\nu'\ll\lambda$, hence with some density $f'_\gamma\in L_\infty(\X,\lambda)_+$ such that $\Vert f'_\gamma\Vert_\infty\leq\gamma$.
From the above argument, $f'_\gamma \leq \gamma\wedge f^*$ a.e. on $\X$ so that
\[0\,=\,\nu'(\X)-\nu^*_\gamma(\X)\,=\,\int_\X \underbrace{(f'_\gamma-(\gamma\wedge f^*))}_{\leq 0}\,d\lambda,\]
which implies that $f'_\gamma=(\gamma\wedge f^*)$, a.e. on $\X$. This in turn implies $\nu'=\nu^*_\gamma$, the desired result.
\end{proof}

\subsection{A hierarchy of semidefinite programs}

With $\X\subset\R^n$, let $\mu,\lambda$ be two Borel measures on $\X$ 
of which we know all moments 
\[\mu_\alpha\,=\,\int \x^\alpha\,d\mu;\quad \lambda_\alpha\,=\,\int \x^\alpha\,d\lambda,
\quad\alpha\in\N^n.\]

\begin{assumption}
\label{ass-1}
Both moment sequences $(\mu_\alpha)$ and $(\lambda_\alpha)$, $\alpha\in\N$,
satisfy Carleman's condition (\ref{carleman}).
\end{assumption}

Let $\gamma>0$ be fixed and for every $d\geq 1$, consider the following optimization problem:
\begin{equation}
\label{sdp-primal}
\begin{array}{ll}
\rho_d=\displaystyle\sup_{\y,\u,\v}&L_\y(1)\\
\mbox{s.t.}&y_\alpha+v_\alpha\,=\,\mu_\alpha,\quad\alpha\in\N^n_{2d}\\
&y_\alpha+u_\alpha\,=\,\gamma\,\lambda_\alpha,\quad\alpha\in\N^n_{2d}\\
&\M_d(\y),\M_d(\u),\M_d(\v)\,\succeq0.
\end{array}
\end{equation}
Equivalently, in matrix form:
\begin{equation}
\label{sdp-primal-matrix}
\begin{array}{ll}
\rho_d=\displaystyle\sup_{\y,\u,\v}&L_\y(1)\\
\mbox{s.t.}&\M_d(\y)+\M_d(\v)\,=\,\M_d(\mu)\\
&\M_d(\y)+\M_d(\u)\,=\,\gamma\,\M_d(\lambda)\\
&\M_d(\y),\M_d(\u),\M_d(\v)\,\succeq0.
\end{array}
\end{equation}
Problem (\ref{sdp-primal}) is a semidefinite program. It is straightforward to check that 
(\ref{sdp-primal}) is a relaxation of (\ref{conv-pb-gamma}) and so $\rho_d\geq \rho_\gamma$ for every $d\in\N$.
Its dual is also a semidefinite program
which reads:
\begin{equation}
\label{sdp-dual}
\begin{array}{ll}
\rho^*_d=\displaystyle\inf_{p,q,\sigma}&\displaystyle\int p\,d\mu+\displaystyle\gamma\,\int q\,d\lambda\\
\mbox{s.t.}&p+q-1=\sigma;\quad p,q,\sigma\in\Sigma[\x]_d.
\end{array}
\end{equation}
\begin{thm}
\label{th-main-sdp}
The semidefinite program (\ref{sdp-primal}) has an optimal solution $(\y^*,\u^*,\v^*)$
and there is no duality gap between (\ref{sdp-primal}) and its dual (\ref{sdp-dual}), i.e., $\rho_d=\rho^*_d$.
\end{thm}
\begin{proof}
First observe that (\ref{sdp-primal}) has the trivial solution $\y=0$ and $(v_\alpha)=(\mu_\alpha)$, $(u_\alpha)=\gamma\,(\lambda_\alpha)$.
From the moment constraints we immediately have:
\[L_\y(x_i^{2d})\,\leq\,\int \x^{2d}_id\lambda;\quad
L_\v(x_i^{2d})\,\leq\,\int \x^{2d}_id\mu;\quad L_\u(x_i^{2d})\,\leq\,\gamma\,\int \x^{2d}_id\lambda,\]
for all $i=1,\ldots,n$. In addition $L_\y(1)\leq\mu_0$, $L_\v(1)\leq\mu_0$ and
$L_\u(1)\leq\gamma\lambda_0$. Let
\begin{equation}
\tau_1:=\max[\mu_0,\max_i[\int\x_i^{2d}d\mu]];\quad
\tau_2:=\gamma\,\max[\lambda_0,\max_i[\int\x_i^{2d}d\lambda]].\end{equation}
By Lasserre and Netzer \cite{lass-netzer}
it follows that 
\begin{equation}
\label{lassnetzer}
\sup_{\alpha\in\N^n_{2d}}\vert y_\alpha\vert\,\leq\,\tau_1;\quad
\sup_{\alpha\in\N^n_{2d}}\vert v_\alpha\vert\,\leq\,\tau_1;\quad
\sup_{\alpha\in\N^n_{2d}}\vert u_\alpha\vert\,\leq\,\tau_2.\end{equation}
Therefore the feasible set of (\ref{sdp-primal}) is compact and so
there exists an optimal solution $(\y^*,\u^*,\v^*)$. In addition, it also follows that
the set of optimal solutions of (\ref{sdp-primal}) is also compact. Therefore
there is no duality gap between (\ref{sdp-primal}) and its dual (\ref{sdp-dual}), that is,
$\rho_d=\rho^*_d$; see for instance Trnovsk\'a \cite{strong}.
\end{proof}

\begin{thm}
\label{th-main-convergence}
Let Assumption \ref{ass-1} hold.
For every $d\geq1$, let $(\y^d,\v^d,\u^d)$ be an arbitrary optimal solution
of (\ref{sdp-primal}) and by completing with zeros, consider 
$\y^d,\u^d$ and $\v^d$ as elements of $\R[\x]^*$.

Then the sequence of triplets $(\y^d,\v^d,\u^d)_{d\in\N}\subset(\R[\x]^*)^3$ converges 
to $(\y^*,\v^*,\u^*)\in (\R[\x]^*)^3$ as $d\to\infty$, that is, for every fixed $\alpha\in\N^n$:
\begin{equation}
\label{th-conv-1}
\lim_{d\to\infty}\,y^{d}_\alpha\,=\, y^*_\alpha;\:\quad\lim_{d\to\infty}\,v^{d}_\alpha\,=\,\mu_\alpha-y^*_\alpha;\quad
\lim_{d\to\infty}u^{d}_\alpha\,=\,\gamma\,\lambda_\alpha-y^*_\alpha.\end{equation}
Moreover, $\y^*$ is the vector of moments of the measure $\nu^*_\gamma\leq\mu$, unique optimal solution of (\ref{conv-pb-gamma}),
with density $f^*_\gamma=(\gamma\wedge f^*)\in L_\infty(\X,\lambda)$ and $\Vert f^*_\gamma\Vert_\infty\leq \gamma$. Similarly
$\v^*$ is the vector of moments of the  measure $\psi^*:=\mu-\nu^*_\gamma$. 
\end{thm}
\begin{proof}
Let $(\y^d,\v^d,\u^d)\in(\R[\x]^*)^3$, $d\in\N$, be as in Theorem \ref{th-main-convergence} and define
the triplet $(\hat{\y}^d,\hat{\v}^d,\hat{\u}^d)\in(\R[\x]^*)^3$, $d\in\N$, by:
\begin{eqnarray*}
\hat{y}^d_\alpha&=&y^d_\alpha/\tau_{1d},\quad\forall\,\alpha:\:2d-1\,\leq\vert\alpha\vert\,\leq\,2d\\
\hat{v}^d_\alpha&=&v^d_\alpha/\tau_{1d},\quad\forall\,\alpha:\:2d-1\,\leq\vert\alpha\vert\,\leq\,2d\\
\hat{u}^d_\alpha&=&u^d_\alpha/\tau_{2d},\quad\forall\,\alpha:\: 2d-1\,\leq\vert\alpha\vert\,\leq\,2d,\end{eqnarray*}
for all $d=1,2,\ldots$, 
where $\tau_{1d},\tau_{2d}$ are defined in (\ref{lassnetzer}) in the proof of Theorem \ref{th-main-sdp}. Therefore
\[\sup_{\alpha\in\N^n}\,\vert \hat{y}^d_\alpha\vert\,\leq\,1;\quad
\sup_{\alpha\in\N^n}\,\vert \hat{v}^d_\alpha\vert\,\leq\,1;\quad \sup_{\alpha\in\N^n}\,\vert \hat{u}^d_\alpha\vert\,\leq\,1,\]
and the sequence $\hat{\y}^d$, $d\in\N$, (as well as sequences $\hat{\v}^d$ and $\hat{\u}^d$, $d\in\N$)
is contained in the unit ball of $\ell_\infty$ (compact and sequentially compact for the weak-$*$ topology $\sigma(\ell_\infty,\ell_1)$). By Banach-Alaoglu
Theorem there exist a subsequence $(d_k)$ and sequences $\hat{\y}^*$, $\hat{\v}^*$ and $\hat{\u}^*$ (each in the unit ball of $\ell_\infty$) such that for each $\alpha\in\N^n$:
\[\lim_{k\to\infty}\,\hat{y}^{d_k}_\alpha \,=\,\hat{y}^*_\alpha;\quad
\lim_{k\to\infty}\,\hat{v}^{d_k}_\alpha \,=\,\hat{v}^*_\alpha;\quad \lim_{k\to\infty}\,\hat{u}^{d_k}_\alpha \,=\,\hat{u}^*_\alpha.\]
Therefore,
\begin{equation}
\label{aux77}
\lim_{k\to\infty}\,y^{d_k}_\alpha \,=\,y^*_\alpha;\quad
\lim_{k\to\infty}\,v^{d_k}_\alpha \,=\,v^*_\alpha;\quad \lim_{k\to\infty}\,u^{d_k}_\alpha \,=\,u^*_\alpha,\end{equation}
with:
\begin{eqnarray*}
y^*_\alpha&=&\hat{y}^*_\alpha\cdot\tau_{1d},\quad\forall\,\alpha:\:2d-1\,\leq\vert\alpha\vert\,\leq\,2d\\
v^*_\alpha&=&\hat{v}^*_\alpha\cdot\tau_{1d},\quad\forall\,\alpha:\:2d-1\,\leq\vert\alpha\vert\,\leq\,2d\\
u^*_\alpha&=&\hat{u}^*_\alpha\cdot\tau_{2d},\quad\forall\,\alpha:\: 2d-1\,\leq\vert\alpha\vert\,\leq\,2d,\end{eqnarray*}
for all $d=1,2,\ldots$.
Next, fix $d\in\N$, arbitrary. From (\ref{aux77}) 
\[0\,\preceq \M_d(\y^*)\preceq\,\M_d(\mu);\quad
0\,\preceq \M_d(\v^*)\preceq\,\M_d(\mu);\quad
0\,\preceq \M_d(\u^*)\preceq\gamma\M_d(\lambda).\]
In particular:
\begin{equation}
\label{aux44}
L_{\y^*}(x_i^{2k})\,\leq\,\int_\X x_i^{2k}\,d\mu;\quad
L_{\v^*}(x_i^{2k})\,\leq\,\int_\X x_i^{2k}\,d\mu;\quad
L_{\u^*}(x_i^{2k})\,\leq\,\gamma \int_\X x_i^{2k}\,d\lambda.\end{equation}
Recall that by Assumption \ref{ass-1} Carleman's condition (\ref{carleman}) holds for
$\mu$ and $\lambda$. Therefore (\ref{aux44}) implies that Carleman's condition also holds for $\y^*$, $\v^*$, and $\u^*$.
Next, as $\M_d(\y^*)\succeq0$, $\M_d(\v^*)\succeq0$ and $\M_d(\u^*)\succeq0$, 
then by  \cite[Proposition 3.5]{lass-book-icp}, $\y^*$, $\v^*$, and $\u^*$ are the respective moment sequences of finite Borel measures 
$\nu$, $\psi$, and $\phi$ on $\X$. In addition,  $\nu$, $\psi$, and $\phi$ 
are moment determinate and since (\ref{th-conv-1}) holds it follows that
\[\nu+\psi\,=\,\mu;\quad \nu+\phi\,=\,\gamma\,\lambda,\]
which shows that $\nu$ is a feasible solution of (\ref{conv-pb-gamma}).
We also have
\[\rho_\gamma\,\leq\,\lim_{k\to\infty} \rho_{d_k}\,=\,\lim_{k\to\infty} L_{\y^{d_k}}(1)\,=\,L_{\y^{*}}(1)\,=\,\nu(\X),\]
which proves that $\nu$ is an optimal solution of (\ref{conv-pb-gamma}). But by Theorem \ref{th-abstract-gamma},
the optimal solution of (\ref{conv-pb-gamma}) is unique. Therefore all accumulation points of
$(\y^{d},\v^{d},\u^{d})$, $d\in\N$, are identical since they are the moment sequences of $\nu^*_\gamma$, $\mu-\nu^*_\gamma$ and
$\gamma\lambda-\nu^*_\gamma$, respectively, that is, (\ref{th-conv-1}) holds.
\end{proof}
The meaning of Theorem \ref{th-main-convergence} is as follows: 
Recall that $\nu^*+\psi^*$ (with $\nu^*\ll\lambda$ and $\psi^*\perp\lambda$) is the (unique) Lebesgue decomposition of $\mu$. Then:

- Either $\nu^*$ has a density $f\in L_\infty(\X,\lambda)_+$ with $\Vert f\Vert_\infty\leq \gamma$
in which case in the limit one obtains all moments of $\nu^*$ and $\psi^*$, or 

- $\nu^*$ does not have a density 
$f^*\in L_\infty(\X,\lambda)_+$ with $\Vert f^*\Vert_\infty\leq \gamma$. In this case, 
$\mu$ can be decomposed into a sum $\nu_1+\nu_2$ where $\nu_1$ has a density $\gamma\wedge f^*\in L_\infty(\X,\lambda)_+$
and $\nu_2=\mu-\nu_1\in M(\X)_+$. In the limit one obtains  all moments of $\nu_1$ and $\nu_2$ (but $\nu_2\not\perp\lambda$, i.e. $\nu_2$ is not singular w.r.t. $\lambda$).

\subsection{Recovering the singular part}

A case of particular interest is when the singular part $\psi^*$ $(\perp\lambda)$ of the Lebesgue decomposition
$\nu^*+\psi^*$ of $\mu$ w.r.t. $\lambda$, is supported on finitely many points. 
Assume that
$\nu^*$ has a density in $L_\infty(\X,\lambda)_+$ with $\Vert f^*\Vert_\infty\leq\gamma$. Then
by Theorem \ref{th-main-convergence}, 
\begin{equation}
\label{singular-part}
\lim_{d\to\infty}\,v^{d}_\alpha\,=\,\mu_\alpha-y^*_\alpha\,=\,\int_\X \x^\alpha\,d\psi^*,\quad\forall\,\alpha\in\N^n,\end{equation}
where $(\y^d,\v^d,\u^d)$ is an optimal solution  of the semidefinite program (\ref{sdp-primal}).

Next, if $\psi^*$ has a finite support, say $m$ points $\x_1,\ldots,\x_m\in \X$, its moment matrix $\M_d(\v^*)$ has (finite) rank $m$
for all $d\geq d_0$, for some $d_0$. 
In particular, ${\rm rank}\,\M_{d_0+1}(\v^*)={\rm rank}\,\M_{d_0}(\v^*)=m$.
By Curto \& Fialkow \cite{curto1,curto2} this property is indeed a certificate that 
$(v^*_\alpha)$, $\alpha\in\N^n_{d_0+1}$ is the truncated moment sequence of a measure supported 
on $m={\rm rank}\,\M_{d_0}(\v^*)$ points of $\R^n$.
There is even a linear algebra procedure to extract the $m$ points $\x_1,\ldots,\x_m$
from the sole knowledge of the finitely many moments $(v^*_\alpha)$, $\vert\alpha\vert\leq d_0+1$; see e.g. Henrion and Lasserre \cite{jbl-henrion}.

\begin{prop}
\label{eigenvalue}
Let $\{\x_1,\ldots,\x_m\}\subset \X$ be the support of the singular part $\psi^*$ in the Lebesgue decomposition
$\nu^*+\psi^*$ of $\mu$ w.r.t. $\lambda$. Let $d_0$ be such that ${\rm rank}\,\M_d(\v^*)=m$ for all $d\geq d_0$,
and let $\eta>0$ be the smallest strictly positive eigenvalue of $\M_{d_0}(\v^*)$. Let
$\v^d$ be part of an optimal solution of (\ref{sdp-primal}).

Then for every fixed $\epsilon>0$, there exists $d_\epsilon>d_0$  such that the 
first respective $s(d_0)-m$ and
$s(d_0+1)-m$  eigenvalues (arranged in increasing order) of $\M_{d_0}(\v^d)$ and $\M_{d_0+1}(\v^d)$
are less than $\epsilon$ and their last respective $m$ eigenvalues 
are larger than $\eta/2$.
\end{prop}
\begin{proof}
Recall that $s_d={n+d\choose n}$ is the size of the moment matrix $\M_d(\v^*)$. 
Let $\eta$ be the smallest strictly positive eigenvalue of $\M_{d_0}(\v^*)$.
The eigenvalues of 
$\M_{d_0}(\cdot)$ and $\M_{d_0+1}(\cdot)$ (arranged in increasing order)
are continuous functions of the entries  and (\ref{singular-part}) holds. So by (\ref{singular-part}), given $\epsilon>0$ there exists $d_\epsilon>0$ such that for every $d\geq d_\epsilon$, the moment matrices $\M_{d_0}(\y^d)$  and $\M_{d_0+1}(\y^d)$ 
have $m$ strictly positive eigenvalues with value larger than $\eta/2$ while their other respective 
 $s_{d_0}-m$ and $s_{d_0+1}-m$ eigenvalues have value smaller than $\epsilon$.
 \end{proof}
 
\subsection*{A practical procedure} So one may propose the following numerical procedure with an {\it \`a priori} fixed integer $p>0$.

\begin{itemize}
\item A threshold $10^{-p}$ is proposed to 
``declare" zero an eigenvalue of $\M_d(\v^d)$ as follows. Compute the eigenvalues of $\M_d(\v^d)$ and check whether
they can be grouped into two disjoint sets $A$ and $B$ such that
\[\sigma \in B \Rightarrow \frac{\sigma}{\theta_A} <10^{-p},\quad \mbox{with }\theta_A= \arg\min\{\sigma: \sigma\in A\}.\]
In view of (\ref{singular-part}) this eventually happens when $d$ is sufficiently large and with $\# A=m$. (However it may happen earlier and with $\# A\neq m$.)
So once one has found such sets $A$ and $B$ then one considers that the rank of $\M_d(\y^d)$ is $\# A$.
\item Once an optimal solution $(\y^d,\v^d,\u^d )$ has been computed, check whether
there is some $k\leq d-1$ such that ${\rm rank}\,\M_k(\v^d)={\rm rank}\,\M_{k+1}(\v^d)$
where ``rank" has the above numerical meaning.
\item If the above rank-condition holds one considers that $\M_{k}(\v^d)$
is the truncated moment sequence of a measure supported on $t:={\rm rank}\,\M_k(\v^d)$ points of $\R^n$.
The extraction procedure in Henrion and Lasserre \cite{jbl-henrion} can be applied and yields $t$ points
$\x_1,\ldots,\x_t$.
\end{itemize}
Of course the rank condition ${\rm rank}\,\M_k(\v^d)={\rm rank}\,\M_{k+1}(\v^d)$ (with $k\leq d-1$) can happen 
earlier than when $d\geq d_0$ (recall that $d_0$ is not known in advance). In this case there is no guarantee that the 
extracted points are indeed the support of $\psi^*$.
\subsection{Examples}
Given two measures $\mu,\lambda$ on $\X\subset\R^n$ and their respective moment sequences
$(\mu_\alpha)$ and $(\lambda_\alpha)$, $\alpha\in\N^n$, with no loss of generality we may and will assume that
$\mu$ is a probability measure (otherwise replace $\mu_\alpha$ with $\mu_\alpha/\mu_0$ for all $\alpha\in\N^n$).

\begin{ex}
\label{ex1}
{\rm The first example is one-dimensional with one atom for the singular part.  Let $\X=[0,1]$, $a,b,c\in \X$, $a<b$, and let $\lambda$ 
be the Lebesgue measure on $\X$. Let 
$\nu_{ab}$ be the probability measure distributed uniformly on $[a,b]\subset \X$ and
\[\mu\,=\,\underbrace{p\,\nu_{ab}}_{\nu^*}+\underbrace{(1-p)\,\delta_c}_{\psi^*},\]
where $\delta_c$ is the Dirac measure at the point $c$ and $p\in (0,1)$ is some fixed scalar.
With $a=0.1$, $b=0.7$, $c=0.4$, and $\gamma=2p$, one solves (\ref{sdp-primal}) with $d=9$ (hence overall we look at moments up to order $18$),
to obtain an optimal solution $(\y^d,\v^d,\u^d)$. The first $5$ (normalized) moments of $\psi^*$ read
\[\left[\begin{array}{ccccc}1.00000 &  0.40000&   0.16000&   0.06400 &  0.02560\end{array}\right].\]
while the first $5$ (normalized) moments of $\nu^*$ read
\[\left[\begin{array}{ccccc}1.00000 &  0.40000&   0.19000&   0.10000 &  0.05602\end{array}\right].\]
In Table \ref{table-1} are displayed the first $5$ ``moments" $v^d_k/v^d_0$, $k=0,\ldots4$, computed in (\ref{sdp-primal})
and the resulting relative errors with those of $\psi^*$, for different values of $p\in (0,1)$.
Similarly in Table \ref{table-11} are displayed the first $5$ ``moments" $y^d_k/y^d_0$, $k=0,\ldots4$, computed in (\ref{sdp-primal})
and the resulting relative errors with those of $\nu^*$ (normalized).
As one may expect, the quality of the approximation is very good for small $p$ and the slightly deteriorates when $p$ increases.

Recall that by Theorem \ref{th-main-sdp} the optimal value $\rho_d$ of (\ref{sdp-primal}) is such that
$\rho_d\to\rho_\gamma$ as $d\to\infty$. However, it is worth noting that $\rho_d$
is not very close to $\rho_\gamma=p$ when $d\leq 10$. So it seems that
the semidefinite hierarchy (\ref{sdp-primal}), $d\in\N$, succeeds well in identifying relatively fast
the support of $\psi^*$ and $\nu^*$, but not so well to obtain their respective masses $p$ and $1-p$.
\begin{table}[!h]
\begin{center}
\begin{tabular}{||c|c|c|c|c||}
\hline
\multicolumn{5}{||c||}{p=0.1}\\
\hline
1.00000&0.3998&   0.1601&   0.0642&   0.0258\\
0\%& 0.05\%&   0.06\%&   0.31\%&   0.76\%\\
\hline
\multicolumn{5}{||c||}{p=0.2}\\
\hline
1.00000& 0.39936 &   0.16009&  0.06435 &   0.02597\\
0\% & 0.15\%&   0.05\% &   0.55\% &  1.45\%\\
\hline
\multicolumn{5}{||c||}{p=0.3}\\
\hline
 1.00000& 0.39861&   0.15982&   0.06434&   0.02606\\
0\% & 0.34\% &0.11\% &   0.54\% &   1.79\%\\
\hline
\multicolumn{5}{||c||}{p=0.4}\\
\hline
1.00000& 0.39785&  0.15979&   0.06462&   0.02639\\
0\% &0.53\%  &0.12\% &   0.96\%&   2.9\%\\
\hline
\multicolumn{5}{||c||}{p=0.5}\\
\hline
1.00000& 0.39662&  0.15956&   0.06484&   0.02672\\
0\% &0.85\%&  0.27\% &   1.29\%&   4.2\%\\
\hline
\multicolumn{5}{||c||}{p=0.6}\\
\hline
1.00000&0.39481&   0.15937&   0.06534&   0.02736\\
0\% & 1.31\% &  0.39\%&   2.05\%&   6.4\%\\
 \hline
\end{tabular}
\end{center}
\caption{Example \ref{ex1}: First $5$ approximate moments of $\psi^*$ (normalized) \label{table-1}}
\end{table}
\begin{table}[!h]
\begin{center}
\begin{tabular}{||c|c|c|c|c||}
\hline
\multicolumn{5}{||c||}{p=0.1}\\
\hline
1.00000&0.4018&   0.1872&   0.0959&   0.05241\\
0\%& 0.45\%&   1.5\%&   4.2\%&   6.8\%\\
\hline
\multicolumn{5}{||c||}{p=0.2}\\
\hline
1.00000& 0.40232 &  0.18775&   0.09640&   0.05269\\
0\%& 0.58\% & 1.2\%&  3.7\% &  6.3\%\\
\hline
\multicolumn{5}{||c||}{p=0.3}\\
\hline
1.00000& 0.40294 &  0.18849&   0.09699&   0.05311\\
0\% & 0.73\% &  0.79\% &  3.09\% &  5.46\%\\
\hline
\multicolumn{5}{||c||}{p=0.4}\\
\hline
1.00000& 0.40288&  0.18837&  0.09688&   0.05303\\
0\% & 0.71\% & 0.86\% & 3.21\% & 5.62\%\\
\hline
\multicolumn{5}{||c||}{p=0.5}\\
\hline
 1.00000& 0.40294&  0.18848&   0.09698&   0.05311\\
0\% & 0.73\% &0.8\% &  3.10\% & 5.47\%\\
\hline
\multicolumn{5}{||c||}{p=0.6}\\
\hline
1.00000& 0.40291&   0.18846&  0.09698&   0.05311\\
 0\% & 0.72\%&  0.81\% & 3.11\%&  5.46\%\\
 \hline
\end{tabular}
\end{center}
\caption{Example \ref{ex1}: First $5$ approximate moments of $\nu^*$ (normalized) \label{table-11}}
\end{table}

}\end{ex}

\begin{ex}
\label{ex2}
{\rm The second  example is also one-dimensional but with two atoms for the singular part.  Let $0<p<1$, $\X=[0,1]$, $a,b,c_1,c_2\in \X$, $a<b$, and let $\lambda$ 
be the Lebesgue measure on $\X$. Let 
$\nu_{ab}$ be the probability measure distributed uniformly on $[a,b]\subset \X$ and
\[\mu\,=\,\underbrace{p\,\nu_{ab}}_{\nu^*}+\underbrace{\frac{(1-p)}{2}\,(\delta_{c_1}+\delta_{c_2})}_{\psi^*}.\]
With $a=0.1$, $b=0.7$, $c_1=0.4$, $c_2=0.5$, and $\gamma=2p$, one solves (\ref{sdp-primal}) with $d=9$ 
(hence overall we look at moments up to order $18$).
The first 5 moments of the measure $(\delta_{c_1}+\delta_{c_2})/2$ read
\[\left[\begin{array}{ccccc}1.00000 &  0.45000&   0.20500&   0.09450 &  0.04405\end{array}\right].\]
In Table \ref{2-atoms-1} one displays the first $5$ approximate ``moments" $(v^d_k/v^d_0)$, $k=0,\ldots,4$, 
and the resulting relative errors with those of $\psi^*$, for various values of $p\in (0,1)$. 
Similarly, in Table \ref{2-atoms-2} one displays the first $5$ approximate ``moments" $(y^d_k/y^d_0)$, $k=0,\ldots,4$, 
and the resulting relative errors with those of $\nu^*$.
Again one observes that the support of $\psi^*$ is relatively well recovered with few moments ($d\leq 10$). However, and
as in  Example \ref{ex1}, the optimal value $\rho_d$ is not very close to $\rho_\gamma$ when $d\leq 10$.
\begin{table}[!h]
\begin{center}
\begin{tabular}{||c|c|c|c|c||}
\hline
\multicolumn{5}{||c||}{p=0.1}\\
\hline
1.00000 &  0.44952&   0.20435&   0.09390&   0.04359\\
0\% &0.11\% &  0.31\% &   0.62\%&   1.04\%\\
\hline
\multicolumn{5}{||c||}{p=0.2}\\
\hline
1.00000&0.44934&   0.20404&   0.09358&  0.04332\\
0\%&0.14\%&   0.46\% &  0.96\%  & 1.63\%\\
\hline
\multicolumn{5}{||c||}{p=0.3}\\
\hline
1.00000&0.44910&   0.20354&   0.09305&   0.04288\\
0\%&0.19\%&   0.71\%   &1.53\%&   2.64\%\\
\multicolumn{5}{||c||}{p=0.4}\\
\hline
1.00000&0.44896&   0.20305&   0.09247&  0.04239\\
0\%& 0.23\%&   0.94\%&   2.14\%&   3.76\%\\
\hline
\multicolumn{5}{||c||}{p=0.5}\\
\hline
1.00000& 0.44894&   0.20250&   0.091741&   0.04173\\
0\%&0.23\%&   1.2\%&   2.91\% &  5.24\%\\
\hline
\multicolumn{5}{||c||}{p=0.6}\\
\hline
1.00000&0.44916&   0.20175&   0.09063&  0.04071\\
0\%&0.18\%  & 1.58\% &  4.09\% &   7.5\%\\
\hline
\end{tabular}
\end{center}
\caption{Example \ref{ex2}: First $5$ approximate ``moments" of $\psi^*$ (normalized) \label{2-atoms-1}}
\end{table}
\begin{table}[!h]
\begin{center}
\begin{tabular}{||c|c|c|c|c||}
\hline
\multicolumn{5}{||c||}{p=0.1}\\
\hline
1.00000& 0.41114&   0.19720&   0.10378&   0.05781\\
0\% & 2.78\% &  3.79\% &3.78\% & 3.20\% \\
\hline
\multicolumn{5}{||c||}{p=0.2}\\
\hline
1.00000& 0.40892&   0.19520&   0.10230&   0.05680\\
0\%& 2.23\% &  2.74\% &2.30\% &1.39\%\\
\hline
\multicolumn{5}{||c||}{p=0.3}\\
\hline
1.00000& 0.40844&  0.19477&  0.10198&  0.05659\\
0\% & 2.11\% &2.51\% &1.98\% &1.02\%\\
\multicolumn{5}{||c||}{p=0.4}\\
\hline
 1.00000& 0.40789&   0.19427&   0.10162&0.05635\\
0\%& 1.97\% &2.24\% &  1.62\% &0.59\%\\
\hline
\multicolumn{5}{||c||}{p=0.5}\\
\hline
1.00000& 0.40742&  0.19384&  0.10131&0.05614\\
0\% & 1.85\%&2.02\% &1.31\% &0.2\%\\
\hline
\multicolumn{5}{||c||}{p=0.6}\\
\hline
1.00000& 0.40693&  0.19342&  0.10101&   0.05594\\
0\%& 1.73\%&1.8\%&1.01\%&  0.13\%\\
\hline
\end{tabular}
\end{center}
\caption{Example \ref{ex2}: First $5$ approximate ``moments" of $\nu^*$ (normalized) \label{2-atoms-2}}
\end{table}

}\end{ex}
\begin{ex}
\label{ex3}
{\rm The third  example is two-dimensional with the singular part $\psi^*$ being uniformly supported on
the point $\x=(1,2)$. Then with $p\in (0,1)$,
\[\mu\,=\,p\,\nu^* +\underbrace{(1-p)\,\delta_{(1,2)}}_{\psi^*};\qquad \nu^*\ll\lambda,\]
where $\nu^*$ is the the (normalized) Gaussian measure with density $\exp(-x_1^2-x_2^2)$, $\lambda=2\nu^*$
and $\gamma=2p$.
With $d=9$, results are displayed in Table \ref{1-atom-normal}  for different values of the weight $p\in (0,1)$.
The first line displays 
the maximum relative error between the computed moment vector $\v^d/v^d_0$ and the moments
of $\psi^*=\delta_{(1,2)}$ up to order $4$. The second line displays the maximum relative error 
between the (normalized) second order moments $\int x_1^2d\nu^*$ and $\int x_2^2d\nu^*$ and their approximation
in the vector $\y^d/y^d_0$ (the first order moments and $\int x_1x_2d\nu^*$ vanish while their approximation in $\y^d/y^d_0$ are less than $0.008$). The third line displays $\rho_d$ that ideally should be close to $p$.

\begin{table}[!h]
\begin{center}
\begin{tabular}{||c|c|c|c|c|c|c|c|c||}
\hline
& p=0.1 &p=0.2&p=0.3&p=0.4&p=0.5&p=0.6 &p=0.7&p=0.8\\
\hline
\hline
$\psi^*$ &0.02\%&    0.05\%   & 0.08\%  &  0.11\%  &  0.22\% &   0.30\%   & 0.41\%   & 0.63\%\\
$\nu^*$ &0.02\%&    0.05\%   & 0.03\%  &  0.04\%  &  0.09\% &   0.04\%   & 0.01\%   & 0.06\%\\
$\rho_9$ &0.1005&    0.2010  &  0.3014  &  0.4019 &   0.5026 &   0.6028  &  0.7033 &   0.8035\\ 
\hline
\end{tabular}
\end{center}
\caption{Example \ref{ex3}: Relative error on the ``moments" up to order $4$ of $\psi^*$ and $\nu^*$ (normalized)\label{1-atom-normal}}
\end{table}
}
\end{ex}
\begin{ex}
\label{ex4}
{\rm The fourth  example is two-dimensional with the singular part $\psi^*$ being uniformly supported on
the two points $(1,2)$ and $(-2,1)$. Then with $p\in (0,1)$,
\[\mu\,=\,p\,\nu^* +\underbrace{(1-p)\,(\delta_{(1,2)}+\delta_{(-2,1)})/2}_{\psi^*};\qquad \nu^*\ll\lambda,\]
where $\nu^*$ is the the (normalized) Gaussian measure with density $\exp(-x_1^2-x_2^2)$, $\lambda=\nu^*$
and $\gamma=2p$.
With $d=9$, results are displayed in Table \ref{2-atom-normal}  for different values of the weight $p\in (0,1)$.
The first line displays 
the maximum relative error between the computed moment vector $\v^d/v^d_0$ and the moments
of $\psi^*=\delta_{(1,2)}$ up to order $4$. The second line displays the maximum relative error 
between the (normalized) second order moments $\int x_1^2d\nu^*$ and $\int x_2^2d\nu^*$ and their approximation
in the vector $\y^d/y^d_0$ (the first order moments and $\int x_1x_2d\nu^*$ vanish while their approximation in $\y^d/y^d_0$ are less than $0.011$). Again the third line displays $\rho_d$ that ideally should be close to $p$.

\begin{table}[!h]
\begin{center}
\begin{tabular}{||c|c|c|c|c|c|c|c|c||}
\hline
& p=0.1 &p=0.2& p=0.3&p=0.4&p=0.5&p=0.6 &p=0.7&p=0.8\\
\hline
\hline
$\psi^*$&0.03\%& 0.06\% &0.08\%& 0.14\% &  0.20\% & 0.26\% &0.44\% &0.71\%\\
$\nu^*$&1.45\%& 1.92\% &2.15\%& 2.11\% &  2.14\% & 2.24\% &2.21\% &2.23\%\\
$\rho_9$& 0.1012 &   0.2019 &   0.3028 &   0.4035 &   0.5040 &   0.6047 &   0.7056 &   0.8062\\
\hline
\end{tabular}
\end{center}
\caption{Example \ref{ex4}: maximum relative error on the ``moments" up to order $4$ of $\psi^*$  and $\nu^*$ (normalized)\label{2-atom-normal}}
\end{table}
}
\end{ex}
\begin{ex}
\label{ex5}
{\rm The fifth  example is two-dimensional with the singular part $\psi^*$ being uniformly supported on
the unit circle $\{\x\in\R^2: x_1^2+x_2^2=1\}$.  Then with $p\in (0,1)$,
\[\mu\,=\,p\,\nu^* +(1-p)\,\psi^*;\qquad \nu^*\ll\lambda,\]
where $\nu^*$ is the normalized Gaussian measure with density $\exp(-x_1^2-x_2^2)$, $\lambda=\nu^*$
and $\gamma=2p$.
With $d=7$ Table \ref{circle-normal} displays the relative error between the moments 
$\int x_1^2d\psi^*$, $\int x_1^4d\psi^*$ and $\int x_1^2x_2^2d\psi^*$ (the odd moments being zero)
and their respective approximation in $\v^d/v^d_0$,
for the value of $p=0.1$, $0.2$, $0.3$ and $0.4$. 
For larger values of $d$ the semidefinite solver encounters numerical difficulties and the numerical output 
cannot be trusted. The last column also displays $L_{\v^d/v^d_0}((x_1^2+x_2^2-1)^2)$ which ideally should
be $\int (x_1^2+x_2^2-1)^2)d\psi^*/\psi^*(\R^2)=0$ since for every $\alpha\in\N^n$, $v^d_\alpha\to \int\x^\alpha d\psi^*$ as $d\to\infty$; recall (\ref{singular-part}) when $\v^d$ is a moment sequence.

Concerning the approximation of the moments of $\nu^*$:
The relative error between the second order moment $\int x_1^2d\nu^*$ (normalized) 
and its approximation (in the computed vector of moments $\y^d/y^d_0$) 
is less than $1\%$ for all $p=0,1$, $0.2$, $0.3$ and $0.4$. On the other hand,
for the order-4 moments $\int x_1^4d\nu^*$ and $\int x_1^2x_2^2d\nu^*$ (normalized) the relative error is about $20\%$.

In Table \ref{circle-box} the same results are displayed but this time when $\nu^*$ is 
uniformly distributed on the unit box $[-1,1]^2$, $\lambda=\nu^*$ and $\gamma=2p$.
In this case the relative error between the second order moment $\int x_1^2d\nu^*$ (normalized) 
and its approximation (in the computed vector of moments $\y^d/y^d_0$) 
is about $11\%$ for all $p=0,1$, $0.2$, $0.3$ and $0.4$.
For the order-4 moment $\int x_1^4d\nu^*$ (normalized) it is about $13\%$ in all cases, and
for the order-4 moment $\int x_1^2x_2^2d\nu^*$ (normalized) it is about $8\%$ in all cases.

So in both Example \ref{ex4} and Example \ref{ex5} the singular part $\psi^*$ (normalized) is 
well recovered even though the absolutely continuous part $\nu^*$ is not so well recovered  
(in particular the moments of order $4$ are not very accurate).

\begin{table}[!h]
\begin{center}
\begin{tabular}{||c|c|c|c||c||}
\hline
& $x_1^2$ & $x_1^4$ & $x_1^2x_2^2$&$L_{\v^d/v^d_0}((x_1^2+x_2^2-1)^2)$\\
\hline
\hline
p=0.1 & 0.19\% &   0.52\% &   0.53\% & 0.001\\
 p=0.2 & 0.47\% &   1.28\% &   1.28\% & 0.003\\
p=0.3&    0.94\% &   2.76\%&    2.76\% & 0.009\\
 p=0.4&  1.87\% &   5.93\%&   5.93\% & 0.02\\
\hline
\end{tabular}
\end{center}
\caption{Example \ref{ex5}: $d=7$; $\nu^*$ Gaussian; relative error on the ``moments" 
$\int x_1^2d\psi^*$, $\int x_1^4d\psi^*$ and $\int x_1^2x_2^2d\psi^*$ of $\psi^*$ 
(normalized)\label{circle-normal}}
\end{table}
\begin{table}[!h]
\begin{center}
\begin{tabular}{||c|c|c|c||c||}
\hline
& $x_1^2$ & $x_1^4$ & $x_1^2x_2^2$&$L_{\v^d}((x_1^2+x_2^2-1)^2)$\\
\hline
\hline 
 p=0.1 & 0.26\% &    0.93\% & 0.61\% &0.002\\
p=0.2&    0.62\% &   2.22\%  &1.47\% &0.0004\\
 p=0.3&  1.15\%&    4.09\% &2.76\%  &0.0008\\
 p=0.4 &   1.87\% &   6.97\% & 5.27\% & 0.0016\\
\hline
\end{tabular}
\end{center}
\caption{Example \ref{ex5}: $d=7$; $\nu^*$ uniform on the unit box; relative error on the ``moments" 
$\int x_1^2d\psi^*$, $\int x_1^4d\psi^*$ and $\int x_1^2x_2^2d\psi^*$ of $\psi^*$ (normalized)\label{circle-box}}
\end{table}
}
\end{ex}
As illustrated in Example \ref{ex5}, numerical problems can be encountered relatively fast (e.g. when $d>7$ in Example \ref{ex5}).
In the next section we briefly discuss how to (partly) address some numerical issues.

\subsection{Some numerical issues}

As already mentioned in introduction, some numerical issues should be taken into account.
Indeed, in the above simple examples one has encountered some numerical difficulties when $d> 10$
(i.e. when moments of order $>20$ appear) and even when $d>7$ for Example \ref{ex5}.
Such numerical problems were not due to the size but rather to the ill-conditioning 
of the semidefinite program (\ref{sdp-primal}) 
(in turn due to the use of the monomial 
basis $(\x^\alpha)_{\alpha\in\N^n}$). Clearly the basis of monomials $(\x^\alpha)$, $\alpha\in\N$, mainly used for modeling and algorithm implantation convenience, is a very bad choice from a numerical viewpoint. This is especially true as we use semidefinite solvers for which such issues can be crucial even for relatively small size matrices, as observed in the examples when $d>10$.

However other (and better choices) of basis are possible. For instance let $(\L_\alpha)$, $\alpha\in\N^n$, be the basis of polynomials orthonormal
w.r.t. $\lambda$. They can be obtained from the moments $(\lambda_\alpha)_{\alpha\in\N^n}$ by simple computation of certain determinants, as described in e.g. Dunkl and Xu \cite{dunkl} and Helton et al. \cite{helton}. In this case one writes
\begin{eqnarray*}
\x\mapsto f(\x)&:=&\hat{f}_\alpha\,\L_\alpha(\x),\qquad \f\in\R[\x],\\
L_\y(f)&=&\sum_{\alpha\in\N^n}\, \hat{f}_\alpha \,y_\alpha,\qquad \forall\,f\in\R[\x],
\end{eqnarray*}
for some vector $(\hat{f}_\alpha)$ of coefficients. In this new basis $(\L_\gamma)_{\gamma\in\N^n}$, 
\[\L_\alpha(\x)\cdot\L_\beta(\x)\,=\,\sum_\gamma c^{\alpha\beta}_\gamma\,\L_\gamma(\x),\quad\alpha,\beta\in\N^n,\]
for some real scalars $(c^{\alpha\beta}_\gamma)$, and the moment matrix reads
\[\overline{\M}_d(\y)(\alpha,\beta)\,:=\,L_\y(\L_\alpha\,\L_\beta)\,=\,\sum_\gamma c^{\alpha\beta}_\gamma\,y_\gamma,\quad\alpha,\beta\in\N^n_d.\]
The advantage in doing so is that the resulting moment matrix $\overline{\M}_d(\lambda)$ in that basis 
is the identity matrix $\I_d$. Hence in the semidefinite program (\ref{sdp-primal-matrix}) the constraint $\M_d(\y)+\M_d(\u)=\M_d(\lambda)$ becomes
$\overline{\M}_d(\y)+\overline{\M}_d(\u)=\I_d$ when the moment matrices 
are expressed in the new basis $(\L_\alpha)_{\alpha\in\N^n_d}$. 

Similarly, let $(\D_\alpha)$, $\alpha\in\N^n$, be the basis of polynomials orthonormal
w.r.t. $\mu$.  Again as for the $(\L_\alpha)$, the $(\D_\alpha)$  can be obtained from the moments $(\mu_\alpha)$, and
\[\D_\alpha(\x)\cdot\D_\beta(\x)\,=\,\sum_\gamma q^{\alpha\beta}_\gamma\,\D_\gamma(\x),\quad\alpha,\beta\in\N^n.\]
for some real scalars $(q^{\alpha\beta}_\gamma)$. The resulting matrix of moments $\tilde{\M}_d(\mu)$ in that basis is also the identity $\I_d$. Of course, the vector of polynomials $\D_d=(\D_\alpha)$
and $\L_d=(\L_\alpha)$, $\alpha\in\N^n_d$, satisfy $\D_d=\Theta_d\L_d$ for some non singular matrix $\Theta_d$. Therefore, when expressed 
in the $(\D_\alpha)$ basis, the previous constraint $\M_d(\y)+\M_d(\v)=\M_d(\lambda)$ now becomes 
$\tilde{\M}_d(\y^T\Theta)+\tilde{\M}_d(\v)=\I_d$ while $\overline{\M}_d(\y)+\overline{\M}_d(\u)=\I_d$.

\subsection*{Acknowledgements}
Research  funded by by the European Research Council
(ERC) under the European Union's Horizon 2020 research and innovation program
(grant agreement ERC-ADG 666981 TAMING)"

\end{document}